\documentclass[11pt,reqno]{amsart}
\usepackage{amssymb,amsfonts,amscd,amsbsy, color}
\usepackage[mathscr]{eucal}
\usepackage{url}

\newtheorem{theorem}{Theorem}[section]
\newtheorem{lemma}[theorem]{Lemma}
\theoremstyle{definition}
\newtheorem*{entry}{\textbf{Entry (K.2)}}
\numberwithin{equation}{section}

\let\eps\varepsilon

\begin{document}

\title[On the (K.2) supercongruence of Van Hamme]
{On the (K.2) supercongruence of Van Hamme}

\author{Robert Osburn}
\address{School of Mathematical Sciences, University College Dublin, Belfield, Dublin 4, Ireland}
\email{robert.osburn@ucd.ie}

\author{Wadim Zudilin}
\address{School of Mathematical and Physical Sciences, The University of Newcastle, Callaghan, NSW 2308, Australia}
\email{wadim.zudilin@newcastle.edu.au}

\address{Max-Planck-Institut f{\"u}r Mathematik, Vivatsgasse 7, D-53111, Bonn, Germany}

\date{\today}


\begin{abstract}
We prove the last remaining case of the original 13 Ramanujan-type supercongruence conjectures due to Van Hamme from 1997.
The proof utilizes classical congruences and a WZ pair due to Guillera. Additionally, we mention some future directions concerning this type of supercongruence.
\end{abstract}

\maketitle

\section{Introduction}

In his second notebook, Ramanujan recorded the following formula for $1/\pi$ (see \cite[p.~352]{berndt}):
\begin{equation} \label{16overpi}
\sum_{n=0}^{\infty} \frac{(\frac{1}{2})_{n}^3}{n!^3} (42n+5) \frac{1}{64^n} = \frac{16}{\pi},
\end{equation}
which he later reproduced in \cite{rama14} together with other similar instances that would revolutionize the history of computing $\pi$ in the 1980's.
Here and throughout, we use the Pochhammer symbol
$(a)_{n} := \Gamma(a+n)/\Gamma(a)$ for the quotient of two gamma functions, so that
$(a)_{0}=1$ and
$(a)_{n} = a (a+1)(a+2) \dotsb (a+n-1)$ if $n$ is a positive integer.

Curiously, a proof for \eqref{16overpi} was not discovered until 1987 \cite{borweins}.
In 1997, Van Hamme conjectured a $p$-adic analogue of \eqref{16overpi}, namely:

\begin{entry}[{Van Hamme \cite{vh}}]
Let $p$ be an odd prime. Then
\begin{equation*}
\sum_{n=0}^{\frac{p-1}{2}} \frac{(\frac{1}{2})_{n}^3}{n!^3} (42n+5) \frac{1}{64^n} \equiv 5p (-1)^{\frac{p-1}{2}} \pmod{p^4}.
\end{equation*}
\end{entry}

Entry (K.2) is one of 13 Ramanujan-type supercongruence conjectures originally due to Van Hamme \cite{vh}.
The other 12 have now been proven using a variety of techniques.
For example, Van Hamme \cite{vh} himself used properties of certain orthogonal polynomials to prove cases (C.2), (H.2) and (I.2).
Kilbourn \cite{kilbourn} applied Greene's hypergeometric series \cite{greene} in order to settle case (M.2),
while McCarthy and the first author \cite{mo} combined this approach with Whipple's transformation to prove (A.2).
Mortenson \cite{mort} then used a similar argument to deal with (B.2).
The second author \cite{zudilin} adopted the method of Wilf--Zeilberger (WZ) pairs to not only give another proof of (B.2),
but demonstrate several new Ramanujan-type supercongruences.
Long \cite{long} utilized a combination of combinatorial identities, $p$-adic analysis and transformations together with ``strange'' evaluations
of ordinary hypergeometric series due to Gosper, Gessel and Stanton to give yet another proof of (B.2) and prove (J.2).
Recently, this strategy has been successfully executed by Swisher \cite{swisher} to also handle cases (E.2), (F.2), (G.2) and (L.2). Finally, Long and Ramakrishna \cite{longrama} showed (D.2) using a ``pertubed'' $_7F_6$ hypergeometric series and Dougall's formula. The purpose of this paper is to prove the last remaining case of Van Hamme's conjectures:

\begin{theorem} \label{main}
Van Hamme's supercongruence \textup{(K.2)} is true.
\end{theorem}

We prove Theorem \ref{main} in Section~\ref{s2} using classical congruences of Wolstenholme and Morley and a WZ pair due to Guillera.
In Section~\ref{s3}, we make some remarks concerning future study.

\section{Proof of Theorem \ref{main}}
\label{s2}

We first require two preliminary results.

\begin{lemma}
\label{key1}
Let $p > 3$ be prime. Then
\begin{equation*}
\prod_{k=1}^{p-1} (p+2k) \equiv (-1)^{\frac{p-1}{2}} \prod_{k=1}^{\frac{p-1}{2}} (2k-1)^2 \pmod{p^3}.
\end{equation*}
\end{lemma}

\begin{proof}
We have
\begin{align*}
\prod_{k=1}^{p-1}(p+2k)
&=\prod_{k=1}^{\frac{p-1}{2}}(2p+(2k-1))(2p-(2k-1))
=(-1)^{\frac{p-1}{2}}\prod_{k=1}^{\frac{p-1}{2}}((2k-1)^2-(2p)^2) \\
&=(-1)^{\frac{p-1}{2}}\prod_{k=1}^{\frac{p-1}{2}}(2k-1)^2\cdot\biggl( 1-(2p)^2 \biggl(\frac1{1^2}+\frac1{3^2}+\dots+\frac1{(p-2)^2} \biggr)+O(p^4)\biggr)
\\
&\equiv(-1)^{\frac{p-1}{2}}\prod_{k=1}^{\frac{p-1}{2}}(2k-1)^2\pmod{p^3}.
\end{align*}
Here we have used that
\begin{equation*}
\frac1{1^2}+\frac1{3^2}+\dots+\frac1{(p-2)^2} = H_{p-1}^{(2)} - \frac{1}{4} H_{\frac{p-1}{2}}^{(2)}
\end{equation*}
and for primes $p>3$ \cite{wol}
\begin{equation}
\label{wellknown}
H_{p-1}^{(2)} \equiv H_{\frac{p-1}{2}}^{(2)} \equiv 0 \pmod{p},
\end{equation}
where $H_{n}^{(i)} := \sum_{j=1}^{n} j^{-i}$ are the generalized harmonic numbers.
\end{proof}

\begin{lemma}
Let $p$ be an odd prime. Then
\begin{equation} \label{key2}
\sum_{k=2}^{p} (-1)^k \frac{(\frac{1}{2})_{\frac{p-1}{2} + k} (\frac{1}{2})_{\frac{p+1}{2} - k}^2}{(1)_{p+1-k}} \equiv 0 \pmod{p^3}.
\end{equation}
\end{lemma}

\begin{proof}
One can directly confirm the desired congruence for $p=3$. Summing over $n=p+1-k$ instead we can write
\begin{align*}
\sum_{k=2}^{p} (-1)^k \frac{(\frac{1}{2})_{\frac{p-1}{2} + k} (\frac{1}{2})_{\frac{p+1}{2} - k}^2 }{ (1)_{p+1-k} }
&=\frac{(1+ \frac{p}{2})_p}{(\frac{1}{2})_{\frac{p+1}{2}}} \sum_{n=1}^{p-1}\frac{(\frac{-p}{2})_n^2}{n!(\frac{-3p}{2})_n}
\\
&=\frac{(1+\frac{p}{2})_p}{(\frac{1}{2})_{\frac{p+1}{2}}}\cdot\frac{(\frac{-p}{2})^2}{\frac{-3p}{2}} \sum_{n=1}^{p-1}\frac{(1-\frac{p}{2})_{n-1}^2}{n!(1-\frac{3p}{2})_{n-1}}
\\
&=-\frac{(1+\frac{p}{2})_p}{(\frac{1}{2})_{\frac{p+1}{2}}}\cdot\frac{p}{6} \sum_{n=1}^{p-1}\frac{(1-\frac{p}{2})_{n-1}^2}{n!(1-\frac{3p}{2})_{n-1}}.
\end{align*}
Thus, it suffices to show
\begin{equation}
\label{easier}
\sum_{n=1}^{p-1}\frac{(1-\frac{p}{2})_{n-1}^2}{n!(1-\frac{3p}{2})_{n-1}}\equiv0\pmod {p^2}
\end{equation}
for $p>3$. Using
$$
(1+\eps)_k=(1+\eps)(2+\eps)\dotsb(k+\eps) =k!\bigl(1+\eps H_k^{(1)} +O(\eps^2)\bigr),
$$
we have
$$
\frac{(1-\frac{p}{2})_{n-1}^2}{n!(1-\frac{3p}{2})_{n-1}}
=\frac{(n-1)!^2\bigl(1- \frac{p}{2} H_{n-1}^{(1)} + O(p^2)\bigr)^2}{n!\,(n-1)!\,\bigl(1- \frac{3p}{2} H_{n-1}^{(1)} + O(p^2)\bigr)}
=\frac1n\,\Bigl( 1+\frac{p}{2} H_{n-1}^{(1)} + O(p^2) \Bigr)
$$
for $n=1,\dots,p-1$. We thus obtain
$$
\sum_{n=1}^{p-1}\frac{(1- \frac{p}{2})_{n-1}^2}{n!(1-\frac{3p}{2})_{n-1}}
\equiv H_{p-1}^{(1)} + \frac{p}{2} \sum_{n=1}^{p-1} \frac{H_{n-1}^{(1)}}{n} \pmod{p^2}.
$$
It remains to use \cite{wol}
$H_{p-1}^{(1)} \equiv 0\pmod{p^2}$ and \eqref{wellknown} for $p>3$ and
$$
2\sum_{n=1}^{p-1} \frac{H_{n-1}^{(1)}}{n} = 2\sum_{1\le k<n\le p-1} \frac1{kn} = \bigl( H_{p-1}^{(1)} \bigr)^2 - H_{p-1}^{(2)}.
$$
This establishes \eqref{easier} and thus \eqref{key2} for $p>3$.
\end{proof}

We can now prove our main result.

\begin{proof}[Proof of Theorem \textup{\ref{main}}]
Recall that a pair of rational functions $F(n,k)$ and $G(n,k)$ form a WZ pair if they satisfy
\begin{equation} \label{wz}
F(n,k-1) - F(n,k) = G(n+1, k) - G(n,k).
\end{equation}
The functions (see \cite{guillera} or \cite{zudilin})
\begin{equation*}
F(n,k)=(84n^2 - 56nk + 4k^2 + 52n - 12k +5) \frac{(-1)^k  (\frac{1}{2})_{n}  (\frac{1}{2})_{n+k} (\frac{1}{2})_{n-k}^2}{2^{4n} (1)_{n}^2  (1)_{2n-k+1}}
\end{equation*}
and
\begin{equation*}
G(n,k)=64n^2 \frac{(-1)^k  (\frac{1}{2})_{n} (\frac{1}{2})_{n+k-1}  (\frac{1}{2})_{n-k}^2}{2^{4n} (1)_{n}^2 (1)_{2n-k}}
\end{equation*}
satisfy \eqref{wz} since after division of both sides by $G(n,k)$, one only needs to check that
\begin{align*}
& -\frac{(84n^2 - 56kn + 108n + 4k^2 - 20k + 21) \bigl(\frac{1}{2} + n-k\bigr)^2}{64n^2 (2n-k+2)(2n-k+1)} \\
&\quad\qquad - \frac{(84n^2 - 56nk + 4k^2 + 52n - 12k + 5) \bigl(n+k - \frac{1}{2}\bigr) \bigl(2n-k+2)}{64n^2 (2n-k+1)(2n-k+2)} \\
&\quad = \frac{\bigl(\frac{1}{2} + n \bigr) \bigl(n+k - \frac{1}{2} \bigr) \bigl(\frac{1}{2} + n-k \bigr)^2}{16n^2 (2n-k+2)(2n-k+1)} -1.
\end{align*}

Summing \eqref{wz} over $n=0,\dots,\frac{p-1}{2}$, we obtain (via telescoping)
\begin{equation} \label{2fequalg}
\sum_{n=0}^{\frac{p-1}{2}} F(n,k-1) - \sum_{n=0}^{\frac{p-1}{2}} F(n,k)
= G\Bigl(\frac{p+1}{2}, k\Bigr),
\end{equation}
where we have used that $G(0,k)=0$. Summing \eqref{2fequalg} over $k=1, \dots, p$, we then obtain
\begin{equation} \label{tele}
\sum_{n=0}^{\frac{p-1}{2}} F(n,0)
= \sum_{n=0}^{\frac{p-1}{2}} F(n,p) + \sum_{k=1}^{p} G\Bigl(\frac{p+1}{2}, k\Bigr)
= F\Bigl(\frac{p-1}{2}, p\Bigr) +  \sum_{k=1}^{p} G\Bigl(\frac{p+1}{2}, k\Bigr),
\end{equation}
where we have used that $F(n,k)=0$ if $2n-k+1<0$ because of the presence of $(1)_{2n-k+1}$ in the denominator.
It now suffices to show
\begin{equation} \label{fclaim}
F\Bigl(\frac{p-1}{2}, p\Bigr) \equiv 6p (-1)^{\frac{p-1}{2}} \pmod{p^4}
\end{equation}
and
\begin{equation} \label{gclaim}
 \sum_{k=1}^{p} G\Bigl(\frac{p+1}{2}, k\Bigr) \equiv p (-1)^{\frac{p+1}{2}} \pmod{p^4}.
\end{equation}

We first consider \eqref{fclaim}. As
\begin{align}
\Bigl(\frac{1}{2} \Bigr)_{\frac{p-1}{2} + p} &= \Bigl(\frac{1}{2} \Bigr)_{\frac{p-1}{2}} \prod_{k=0}^{p-1} \Bigl( \frac{p}{2} + k \Bigr),
\nonumber \\
\label{p2}
\frac{(\frac{1}{2})_{\frac{p-1}{2}}}{(1)_{\frac{p-1}{2}}^2} &= \frac{\binom{p-1}{\frac{p-1}{2}}}{2^{p-2} (p-1) (1)_{\frac{p-3}{2}}}
\end{align}
and for primes $p>3$ \cite{morley}
\begin{equation} \label{p3}
\binom{p-1}{\frac{p-1}{2}} \equiv (-1)^{\frac{p-1}{2}} 2^{2p-2} \pmod{p^3},
\end{equation}
we have
\begin{align}
F\Bigl(\frac{p-1}{2}, p\Bigr)
& = 3p^2 \Biggl( \frac{(\frac{1}{2})_{\frac{p-1}{2}} (\frac{1}{2})_{\frac{p-1}{2} + p} (\frac{1}{2})_{\frac{p-1}{2} - p}^2}{2^{2(p-1)} (1)_{\frac{p-1}{2}}^2} \Biggr) \nonumber \\
& \equiv 3p^2 \Biggl( \Bigl(\frac{1}{2}\Bigr)_{\frac{p-1}{2} - p}^2 \prod_{k=0}^{p-1}  \Bigl( \frac{p}{2} + k \Bigr) \Biggr) \pmod{p^4} \nonumber \\
& \equiv 6p \displaystyle \prod_{k=1}^{p-1} (p+2k) \cdot \Biggl( \prod_{k=1}^{\frac{p-1}{2}} (2k-1)^2 \Biggr)^{-1} \pmod{p^4}.
\label{fstep2}
\end{align}
Here, we have used that
$$
(a)_{-n} = \prod_{k=1}^{n} \frac{1}{a-k}.
$$
Thus, \eqref{fclaim} follows from \eqref{fstep2} and Lemma \ref{key1}. We now use \eqref{p2} and \eqref{p3} to obtain
\begin{align}
G\Bigl(\frac{p+1}{2}, k\Bigr)
& = \frac{32p}{2^{2p+2}} (-1)^k \Biggl( \frac{(\frac{1}{2})_{\frac{p-1}{2}}}{(1)_{\frac{p-1}{2}}^2} \Biggr)
\frac{(\frac{1}{2})_{\frac{p+1}{2} + k -1} (\frac{1}{2})_{\frac{p+1}{2} - k}^2}{(1)_{p+1-k}} \nonumber \\
& \equiv \frac{p (-1)^{\frac{p-1}{2}+k}}{2^{p-2} (1)_{\frac{p-1}{2}}} \frac{(\frac{1}{2})_{\frac{p-1}{2} + k} (\frac{1}{2})_{\frac{p+1}{2} - k}^2}{(1)_{p+1-k}} \pmod{p^4}.
\label{gstep1}
\end{align}
So, \eqref{p2}, \eqref{p3} and \eqref{gstep1} imply
\begin{equation} \label{gstep2}
G\Bigl(\frac{p+1}{2}, 1 \Bigr) \equiv (-1)^{\frac{p+1}{2}} p \pmod{p^4}.
\end{equation}
Summing \eqref{gstep1} over $k=2, \dots, p$, then applying Lemma \ref{key2} and \eqref{gstep2} yields \eqref{gclaim}.
The result now follows from \eqref{tele}--\eqref{gclaim} and checking the $p=3$ case.
\end{proof}

\section{Concluding remarks}
\label{s3}

It is still not known if there exists a general framework which explains this type of supercongruence.
Such a theory is especially desirable both as it appears that all known Ramanujan-type series for $1/\pi^a$, $a \geq 1$,
have $p$-adic analogues and all of Van Hamme's original 13 conjectures have extensions.
For example, it has been recently conjectured in~\cite{swisher} that if we let $S(N)$ denote the sum in \eqref{16overpi} truncated at $N$, then
$$
S\Bigl( \frac{p^r - 1}{2} \Bigr) \equiv p (-1)^{\frac{p-1}{2}} S \Bigl(\frac{p^{r-1} - 1}{2} \Bigr) \pmod {p^{4r}}
$$
for all primes $p>2$ and integers $r \geq 1$; for a list of these conjectural extensions, please see \cite{swisher}. This pattern continues as for the conjectural evaluation \cite{zudilinwind}
\begin{equation} \label{gourevich}
\sum_{n=0}^{\infty} \frac{(\frac12)_n^7}{n!^7} (168n^3+76n^2+14n+1)\frac1{2^{6n}} = \frac{32}{\pi^3},
\end{equation}
we also expect the monstrous supercongruence
$$
\tilde S\Bigl( \frac{p^r - 1}{2} \Bigr) \equiv p^3 (-1)^{\frac{p-1}{2}} \tilde S\Bigl(\frac{p^{r-1} - 1}{2} \Bigr) \pmod {p^{8r}}
$$
to be true for primes $p>2$, $p\ne5$ (for $p=5$, replace $p^{8r}$ with $p^{8r-1}$) where $\tilde S(N)$ is the sum in \eqref{gourevich} truncated at $N$.
This is part of a general phenomenon that has to be understood.

\section*{Acknowldegements}
The first author would like to thank both the Institut des Hautes \'Etudes Scientifiques and the Max-Planck-Institut f\"ur Mathematik for their support during the preparation of this paper.
This material is based upon work supported by the National Science Foundation under Grant no.~1002477.
He also thanks Jes\'us Guillera for his helpful comments and suggestions.
The second author acknowledges the support by the Max-Planck-Institut f\"ur Mathematik.

\end{document}